\documentclass[12pt]{amsart}

\DeclareMathOperator{\cof}{\cof}

\usepackage[english]{babel}
\usepackage[utf8]{inputenc}
\usepackage{amsmath}
\usepackage{amssymb}
\usepackage{amsfonts}
\usepackage{amsthm}
\usepackage{mathrsfs}
\usepackage[all]{xy}
\usepackage[pdftex]{graphicx}
\usepackage{color}
\usepackage[pdftex,dvipsnames]{xcolor} 
\usepackage{cite}
\usepackage{url}
\usepackage{indent first}
\usepackage[labelfont=bf,labelsep=period,justification=raggedright]{caption}
\usepackage[english]{babel}
\usepackage[utf8]{inputenc}
\usepackage[colorlinks=true, allcolors=blue]{hyperref}
\usepackage[colorinlistoftodos]{todonotes}
\usepackage{tkz-fct}
\usepackage{mathdots}
\usepackage{comment}
\usepackage{float}

\usepackage{algpseudocode}
\usepackage{mathrsfs}
\usepackage{bbm}
\usepackage{dsfont}
\usepackage{theoremref}
\usepackage[dvipsnames]{xcolor}

\usetikzlibrary{automata,positioning}

\theoremstyle{plain}

\theoremstyle{plain}
\newtheorem{thm}{Theorem}

\newtheorem{lemma}[thm]{Lemma}
\newtheorem{cor}[thm]{Corollary}
\newtheorem{conj}[thm]{Conjecture}
\newtheorem{prop}[thm]{Proposition}

\newtheorem{qn}[thm]{Question}

\theoremstyle{definition}
\newtheorem{defn}[thm]{Definition}

\theoremstyle{remark}
\newtheorem{rem}[thm]{Remark}

\numberwithin{equation}{section}
\numberwithin{thm}{section}
\usepackage[margin=1in]{geometry} 
\usepackage[T1]{fontenc}
\usepackage[utf8]{inputenc}

\usepackage{tgtermes}

\numberwithin{equation}{section}
\numberwithin{thm}{section}

\begin{document}
\title[Ergodic with Isometric Coefficients]{Nonsingular transformations that are ergodic with isometric coefficients and not weakly doubly ergodic}
\date{\today}

\author[Haddock]{Beatrix Haddock}
\address[Beatrix Haddock]{Department of Applied Mathematics\\ University of Washington\\ Seattle, WA, 98195, USA.}
\email{beatrixh@uw.edu}

\author[Leng]{James Leng}
\address[James Leng]{Department of Mathematics\\
     UCLA \\ Los Angeles, CA 90095, USA.}
\email{jamesleng@math.ucla.edu}

\author[Silva]{Cesar E. Silva}
\address[Cesar E. Silva]{Department of Mathematics\\
     Williams College \\ Williamstown, MA 01267, USA.}
\email{csilva@williams.edu}


\maketitle
\begin{abstract}
We study two properties of nonsingular and  infinite measure-preserving ergodic systems: weak double ergodicity, and ergodicity with isometric coefficients. We show that there exist infinite measure-preserving transformations that are ergodic with isometric coefficients but are not weakly doubly ergodic; hence these two notions are not equivalent for infinite measure. We also give type $\text{III}_\lambda$ examples of such systems, for $0<\lambda\leq 1$. We prove that under certain hypotheses, systems that are weakly mixing are ergodic with isometric coefficients and along the way we give an example of a uniformly rigid topological dynamical system along the sequence $(n_i)$ that is not measure theoretically rigid along $(n_i)$ for any nonsingular ergodic finite measure. 
\end{abstract}

\section{Introduction}

   The notion of weak mixing for finite measure-preserving transformations has several different 
   and equivalent characterizations, and each has played different roles in various applications
   of the weak mixing property. It is now well-known that in infinite measure the theory is quite different.
   The first example showing this was given by Kakutani and Parry when they constructed in \cite{KaPa63} 
   infinite measure-preserving Markov shifts $T$ such that $T\times T$ is ergodic but $T\times T\times T$ is not ergodic (as is well-known, this cannot happen for finite measure-preserving transformations). Since then many related notions an examples have been constructed for infinite measure-preserving and nonsingular transformations; we refer to \cite{AdSi18} for a survey of these results and to \cite{GlWe16} where many of these results are discussed in the context 
   of group actions. \\\\
While our first examples are infinite measure-preserving transformations we  also consider a nonsingular versions of our construction. \\\\
\textbf{Acknowledgments:} The research for this paper started during the 2018 SMALL undergraduate research project at Williams College, where the first-named author was part of the ergodic theory group. We thank the other members of the ergodic theory group, which included Hindy Drillick, Alonso Espinosa-Dominguez, Jennifer N. Jones-Baro, and Yelena Mandelshtam. Support for the project was provided by National Science Foundation grant DMS-1659037 and the Science Center of Williams College. The  first-named author was also supported by Williams College in summer 2019.  We would like to thank Terrence Adams, and an anonymous referee for comments and suggestions on our work. Part of this work, in particular Section \ref{rigidity} is based on the   undergraduate thesis at Williams College of BH, supervised by the third-named author.   From  August 2019 until August 2021, CS has been serving as a Program Director in the Division of Mathematical Sciences at the National Science Foundation (NSF), USA, and as a component of this job, he received support from NSF for research, which included work on this paper. Any opinions, findings, and conclusions or recommendations expressed in this material are those of the authors and do not necessarily reflect the views of the National Science Foundation.

\section{Preliminaries}  We let $(X,\mathcal B, \mu)$ denote a standard Borel space that we will assume to be  nonatomic, where $\mu$ is a $\sigma$-finite measure, often assume to be a probability measure; sometimes we may omit the $\sigma$-algebra $\mathcal B$ from the notation. 
If we start with a $\sigma$-finite infinite measure $(X,\mu)$ we can always  choose  an equivalent probability measure $\nu$. In all cases we assume the measures are nonatomic. A transformation $(X,\nu,T)$ (sometimes we may simply write the transformation  $T$) is nonsingular if it is measurable ($T^{-1}(A)\in \mathcal B$ for all $A\in\mathcal B)$ and $\mu(A)=0$ if and only if $\mu(T^{-1}(A))=0$. A transformation $T$ is ergodic if whenever $\mu(T^{-1}A\bigtriangleup A)=0$ we have $\mu(A)=0$ or $\mu(A^c)=0$, and it is conservative if for all sets $A$ of positive measure there exists an integer $n>0$ with $\mu(A\cap T^{-n}A)>0$.
We will assume our transformations are invertible (i.e., invertible on a set of full measure and with measurable inverse); in this case, as the measures are nonatomic, it is known  that if a transformation is ergodic, then it is conservative, see e.g. \cite[3.9.1]{Si08}. \\\\
A nonsingular   transformation $(X,  \mu, T)$
 is \textbf{doubly ergodic} (DE) if $T \times T$ is ergodic.
It  is \textbf{weakly doubly ergodic} (WDE) if for any two measurable sets $A$, $B$ with positive measure, there exists $n \in \mathbb{Z}$ such that $\mu(T^nA \cap A) > 0$ and $\mu(T^nA \cap B) > 0$ (we note that this notion was originally called doubly ergodic  in  \cite{BoFiSi01}, and that we can choose $n>0$ in the definition).
 The transformation $T$ is \textbf{ergodic with isometric coefficients} (EIC---see \cite{GlWe16}) if for any   separable metric space $Y$ and $S : Y \to Y$ an invertible isometry,  for any equivariant Borel  map  $\phi : (X, \nu,T) \to Y$,  i.e., $S \circ \phi = \phi \circ T$ $\mu$-a.e., the image of almost all points in $X$ under $\phi$ is a single point., i.e., the map $\phi$ is constant a.e. 
  A transformation $T$ is \textbf{weakly mixing} (WM) if for any ergodic   measure preserving transformation on a probability space $(Y, m,S)$, the transformation $T \times S$ is ergodic on $X \times Y$ with respect to the product measure $\mu \otimes m$. 
  All these notions only depend on the measure class of the measure. We use the following result from  in \cite{AaLiWe79}.
\begin{thm}
A nonsingular dynamical system $(X, T, \mathfrak{B}, \mu)$ is {weakly mixing} if and only if whenever  $f \in L^\infty(X,\mathbb C)$, $\lambda\in\mathbb C$, are such that $f \circ T = \lambda f$ a.e., then $f$ is constant a.e. (and $\lambda = 1$).
\end{thm}

From the definitions it follows  that if a transformation is DE then it is WDE; the fact that the converse 
is not true for  infinite measure-preserving transformations was shown in \cite{AdFrSi97,BoFiSi01}. Glasner and Weiss \cite{GlWe93} proved that DE implies EIC and that  EIC implies WM, and later Loh and Silva proved that WDE implies EIC (these results are in the context of nonsingular group actions).  In  \cite{GlWe93}, it is shown that for nonsingular group actions, WM implies ergodic; for the case of nonsingular transformations $T$ WM implies WM of $T^n, n\neq 0$ \cite[2.1]{MRS99}. For completeness we observe that EIC implies ergodicity: if a nonsingular transformation $T$ is EIC and $A$ is a $T$-invariant set we can define a map $\phi:X\to\{0,1\}$ by $\phi(x)=0$ for $x\in A$ and $\phi(x)=1$ for $x\in A^c$. Let $S$ be the identity on $\{0,1\}$. Then $\phi$ is a non-constant a.e. equivariant map onto the isometry $S$ on $\{0,1\}$; total ergodicity of $T$ can also be shown by a similar argument.

Glasner and Weiss \cite{GlWe93} also showed that in infinite measure EIC does not imply DE.  In this paper we prove that in infinite measure EIC does not imply WDE.   Glasner and Weiss \cite{GlWe16} also asked whether WM implies EIC, which remains open as far as we know; in Section \ref{rigidity}  we study some properties related to this question. In summary, we now know the following implications and that the converse does not hold for each of the first two implications.  

\[DE \implies WDE \implies EIC \implies WM.\]


\subsection{Rank-One Transformations}
\hfill \break
Rank-one transformations play a central role in ergodic theory and this concept is central to our paper. We describe the cutting and stacking construction of rank-one transformations.\\\\
A \textbf{column}  is a finite ordered collection of intervals of finite measure $$C_k=\{C_k(0), C_k(1),\dots,C_k(h_k-1)\}.$$   $C_k(0)$ is called the \textbf{base} of the $k$th column, $C_k(i)$ is the $i$-th \textbf{level} of the $k$th column, and $h_k$ is the \textbf{height} of the $k$th column; the \textbf{width} of the column $w(C_k)$ is the length of the largest interval in $C_k$.  In the measure-preserving case all the intervals of a given column are of the same length, in the nonsingular case they may be of different lengths. Each column defines a column map $T_{C _k}$
that consists of the (unique orientation preserving) affine translation (simply a translation in the measure-preserving case) that sends interval  $C_k(i)$ to interval $C_{k}(i + 1)$, for $0\leq i<h_k-1$; this map is defined on every level of $C_k$ except for the top level. To specify a rank-one construction we are given $(r_k)$, a sequence of integers with $r_k\geq 2$, and   $(s(k,i)), i=0,\ldots,h_k-1$,  a doubly indexed sequence of nonnegative integers. Let column $C_0$ consist of a single interval.  For each $k\geq 0$, column $C_{k+1}$  is obtained from column $C_{k}$ by first subdividing each interval in $C_k$  into $r_k$ subintervals to obtain $r_k$ subcolumns; denote them $C_{k,0}, \ldots, C_{k,r_k-1}$.  Over each subcolumn $C_{k,i}$    place $s(k,i)$  new intervals to obtain a new subcolumn $C_{k,i}^\prime$, now height $h_k+s(k,i)$, for each $i=0,\ldots,r_k-1$. Finally stack  the new subcolum $C_{k,i+1}^\prime$    right over subcolumn  $C_{k,i}^\prime$ for $i=0,\ldots,r_k-1$  to obtain column $C_{k+1}$; i.e, map the top level in $C_{k,i}^\prime$ to the bottom level in $C_{k,i+1}^\prime$. Note that the column map $T_{C_{k+1}}$ agrees with $T_{C_k}$ wherever $T_{C_k}$ is defined. We let $X$ be the union of all the columns; the new subintervals can be chosen so that $X$ is a interval, which can be of finite or infinite measure. We let $T$ be the limit of the $T_{C_k}$.
Regarding the cuts into $r_k$ subintervals we note that in the measure-preserving case the cuts are uniform, i.e., all the subintervals are of the same length, and in the nonsingular case one must specify in advance the ratios on the subintervals in the cuts but they are done so that the width of the column $w(C_k)$ decreases to 0 as $k\to\infty$.
  
It follows that  for every measurable set of finite measure $A\subset X$ and every $\varepsilon>0$ there exists  $N>0$ such that for every $j\geq N$ there exists a union $\widehat{C}_j$ of some levels of $C_j$ such that
$$\mu(\widehat{C}_j\triangle A)<\varepsilon.$$
One can show that $T$ is nonsingular and ergodic (see \cite[4.2.1]{RuSi89} for a proof that includes the nonsingular case).

 Famous examples of cutting and stacking transformations include the Chac\'on ($r_k=3, s(k,0)=0, s(k,1)=1,s(k,2)=0,$ for all $k\geq 0$) and Kakutani ($r_k=2, s(k,0)=s(k,1)=0$, for all $k\geq 0$) transformations. See e.g. \cite{Si08} for a discussion of these and other examples. A \textbf{rank-one tower transformation} is a rank-one transformation with one cut ($r_k=2$) and whose spacers are only in the last column ($s(k,0)=0$). An example of such a transformation is the Hajian-Kakutani  skyscraper transformation ($s(k,1)=2h_k$), see e.g. \cite{Si08}.

\section{A Construction that is EIC but not WDE}  We have already mentioned that WDE implies EIC; in this section we show that the converse of this implication is not true; our approach is to develop some techniques to verify the EIC property for some transformations. 

\begin{thm}\thlabel{th1}
There exists a measure-preserving rank-one transformation $T$ on a $\sigma$-finite measure space that is not weakly doubly ergodic but is ergodic with isometric coefficients.
\end{thm}
The example is a rank-one transformation, first defined by Adams, Friedman and Silva in \cite{AdFrSi97}, that has been called the HK(+1) transformation since it is a variant of the Hajian-Kakutani transformation. It is a rank-one tower transformation. For the base case choose column $C_0$ consisting of the unit interval $[0,1)$, and    let  $r_k=2, s(k,0)=0$ and $s(k,1)=2h_k+1$, for all $k\ge 0$.  So the height is given by $h_{k+1}=4h_k+1$. This defines an infinite measure-preserving rank-one transformation. In \cite{AdFrSi97}, it was shown to be weakly mixing but not doubly ergodic; the proof in \cite{AdFrSi97} that the transformation is not doubly ergodic also works to show it is not weakly doubly ergodic, and another proof is given in \cite{BoFiSi01}. For reference we recall that the Hajian-Kakutani transformation is obtained when at stage $k$ we add $2h_k$ subintervals instead of $2h_{k}+1$ subintervals; this transformation has eigenvalues and is not weakly mixing, in fact it is not totally ergodic (its square  is not ergodic). In this section we show that the HK(+1) transformation is EIC.  We need the following lemma.

\begin{lemma}\thlabel{dense}
Let  $(X, T, \mu, \mathcal{B})$ be an ergodic nonsingular transformation and $(Y, d)$ a separable metric space with $S$ an invertible isometry on $Y$.  Suppose that there is  a Borel  equivariant map $\phi: X\to Y$, i.e., a Borel map such that $S \circ \phi = \phi \circ T$  $\mu$-a.e.
Then  there exists a point $x \in X$ such that $\phi(x)$ has a dense orbit in a set full measure that is in  the image of $\phi$. 
\end{lemma}
\begin{proof}
By Lemma 1 of \cite{Eigenoperators}, there exists an element $x\in X$ such that for all  $\varepsilon>0$ we have $\mu(\phi^{-1}(B(\phi(x),\varepsilon))>0$. By ergodicity of $T$, as $T$ is also conservative,  $\bigcup_{n>0}T^{n}\phi^{-1} B(\phi(x),\varepsilon)$ is a set of full measure in $X$. We observe that 
\[T^{n}\phi^{-1}B(\phi(x),\varepsilon) =\phi^{-1}S^{n}B(\phi(x),\varepsilon)\]
Thus  $\bigcup_{n>0}S^{n}(B(\phi(x),\varepsilon)=Y$ $\mod \mu\phi^{-1}$. 
Since $S$ is an invertible isometry 
it follows that $S^{n} B(\phi(x), \varepsilon) = B(S^{n}\phi(x), \varepsilon)$. 
Let 
\[Z = \bigcap_{m>0} \bigcup_{n > 0} B(S^{n} \phi(x), 2^{-m}).\]
Then $Z^c = \bigcup_{m>0} [\bigcup_{n > 0} B(S^{n} \phi(x), 2^{-m})]^c$ is a countable union of measure zero sets, so $Z$ has full measure, and we can see that $\{S^n \phi(x):n>0\}$ is dense in $Z$.

The union $\bigcup_{i\in\mathbb Z} B(S^i \phi(x),\varepsilon)$ is the entire space but for a set of measure $0$. Taking the intersection over a sequence of $\varepsilon$ tending to $0$, we have a full measure space for which the orbit of $\phi(x)$ is dense in the space.
\end{proof}

\begin{proof}[Proof of Theorem~\ref{th1}]
Let $\mathcal{X} = (X, T, \mu, \mathcal{B})$ denote the HK(+1) transformation and suppose that there exists a Borel measurable equivariant map $\phi$ from $\mathcal{X}$ to a separable isometry $(Y, d, S)$. 
Choose a point $y=\phi(x)$, for $x\in X$, that has a dense orbit under $S$ in the image of $\phi$ in $Y$. There exists a point $q \in Y$ such that the set
\[A = \{x \in X: d(\phi(x), q) < \frac{\epsilon}{2} \}\]
has positive measure. Let $L$ be a level in column $n$ for which $\mu(A \cap L) > \frac{9}{10}\mu(L)$. Since 
\[\mu(T^{h_n}L \cap L) \ge \frac{1}{2}\mu(L),\text{ we must have }\mu(T^{h_n}A \cap A) > 0.\]
 Thus, there exists some $z \in A$ such that $T^{h_n}z\in A$. So
  \begin{align*}
  d(S^{h_n}(\phi(z)), \phi(z)) &\le d(\phi(z), q) + d(S^{h_n}(\phi(z)), q) \\
  &= d(\phi(z), q) + d ( \phi ( T^{h_n} z ), q) \\&< \frac{\epsilon}{2} + \frac{\epsilon}{2} = \epsilon.
  \end{align*} Next, choose $\ell$ such that 
\[
d(S^\ell(y), \phi(z)) < \epsilon.\] We then have 
\begin{align*}
d(S^{h_n} y, y) &= d(S^{h_n + \ell}y, S^\ell y) \\
&\le d(S^{h_n + \ell}y, S^{h_n}\phi(z)) + d(S^{h_n}\phi(z), \phi(z)) + d(\phi(z), S^\ell y)\\
 &< 3\epsilon.\end{align*}  Now for each $m \ge n$ there exists a level $L'$ in $C_m$  such that $\mu(A \cap L') > \frac{9}{10}\mu(L')$. As a result, we must have $d(S^{h_m}y, y) < 3\epsilon$ for all $m \ge n$. Then we have $d(S^{4h_n + 1}y, y) < 3\epsilon$, so by the triangle inequality, we have
  \begin{align*}
  d(S^{4h_n}y, y) &\le d(S^{h_n}y, y) + d(S^{2h_n}y, S^{h_n}y) + d(S^{3h_n}y, S^{2h_n}y) + d(S^{4h_n}y, S^{3h_n}y) \\&< 12\epsilon.\end{align*}
So
\begin{align*}
d(Sy, y) &\le d(S^{4h_n+1}y, Sy) + d(S^{4h_n + 1}y, y)\\&=  d(S^{4h_n}y, y) + d(S^{4h_n + 1}y, y) < 15\epsilon.
 \end{align*}
This proves that $y$ is a fixed point under $S$, so as the orbit of $y$ is dense the  map $\phi$ is trivial. Thus the HK(+1) transformation is EIC.

Since it was shown in \cite{AdFrSi97,BoFiSi01} that the HK(+1) transformation is not weakly doubly ergodic, this completes the proof that EIC and  WDE are not equivalent. 
\end{proof}

\section{Type $\text{III}_\lambda$ examples}

Let $T$ denote an invertible ergodic nonsingular transformation on a $\sigma$-finite measure space $(X,\mu)$. Let $\omega_n =\omega_n^\mu= \frac{d\mu \circ T^n}{d\mu}$. Define the \textbf{ratio set}
$$r(T) = \{t\in [0,\infty): \forall A, \mu(A) > 0, \exists n > 0, \mu(A \cap T^{-n}A \cap \{x: \omega_n(x) \in B_\epsilon(t)\}) > 0\},$$ where $B_\epsilon(t)=\{x\in [0,\infty): |x-t|<\epsilon\}$.
If $\nu$ is a $\sigma$-finite measure equivalent to $\mu$ there exists a positive a.e measurable function $h$ such that $\nu=h\mu$, so $\omega^\nu=(h\circ T/h)\omega^\mu$, and since every set $A$ of positive measure has a measurable subset of positive measure where $h$ is almost constant, it follows that the ratio  set is the same for all equivalent measures.  It is well known that $r(T) \setminus \{0\}$ is a multiplicative subgroup of the positive real numbers, see e.g.  \cite{DaSi09}. If $r(T) \neq \{1\}$, then $T$ admits no equivalent $\sigma$-finite $T$-invariant measure and we say that it is type III. There are several types of type III transformations:
\begin{itemize}
\item[1.] Type $\text{III}_\lambda$: $r(T) = \{\lambda^n: n \in \mathbb{Z}\} \cup \{0\}$ for some $\lambda \in (0, 1)$
\item[2.] Type $\text{III}_0$: $r(T) = \{0, 1\}$.
\item[3.] Type $\text{III}_1$: $r(T) = [0, \infty)$.
\end{itemize}
In this section, we will extend our construction above to a type $\text{III}_\lambda$ example of a system that is EIC but not WDE. \\\\
Let $\lambda$ be a real number  with  $0<\lambda <1$. We will construct  $\text{III}_\lambda$ nonsingular tower transformations that we will call nonsingular HK$(+1,\lambda)$ transformations,  via cutting and stacking as follows:
\begin{itemize}
\item[1.] Start with the unit interval $[0, 1)$ as our $0$th column: $C_0 = \{[0, 1)\}$.
\item[2.] The $k + 1$th column $C_{k + 1}$ is obtained by cutting intervals of $C_k$ into two in a way such that the ratio of the measure of the two pieces is $\lambda$, with the left piece being larger than the right piece. Then insert $2h_k + 1$ spacers on top of the stack. So $r_k=2, s(k,0)=0, s(k,1)=2h_k+1$, as before, except that now intervals are cut in the uneven ratio $\lambda$.
\end{itemize}
Let $(X, S_\lambda, \mathcal{B}, \mu)$ denote the rank-one system we constructed above; when clear from the context we will write $S$ for $S_\lambda$. We  first show that the system is indeed a type $\text{III}_\lambda$ system.

\begin{thm}
Let $0<\lambda <1$.
$(X, S_\lambda, \mathcal{B}, \mu)$ is a type $\text{III}_\lambda$ ergodic system.
\end{thm}
\begin{proof} Write $S=S_\lambda$.
On a level set, the Radon-Nikodym derivative $\frac{d\mu \circ S}{d\mu}$ on that set is the amount of expansion or contraction of the interval when it is mapped to the level above it. One can verify that the factor of expansion or contraction is always a power of $\lambda$, so the ratio set $r(S)$ is a subset of $\{\lambda^n\} \cup \{0\}$. It thus suffices to show that $r(S)$ contains $\lambda$. Given a measurable set $A$ with positive measure, for $\epsilon > 0$ small, pick a level $I$ on $C_k$ such that $\mu(A \cap I) \ge (1 - \epsilon)\mu(I)$. Then on the left sublevel $J$ of $I$, $\mu(A \cap J) \ge (1 - \epsilon(1 + \lambda)) \mu(J)$. It thus follows that
$$\mu(S^{h_k} A \cap (I \setminus J)) \ge \lambda (1 - \epsilon(1 + \lambda)) \mu(J) = (1 - \epsilon(1 + \lambda)) \mu(I \setminus J)$$
as $I \setminus J$ is the right subinterval of $I$. In addition, we have $\mu(A \cap (I \setminus J)) \ge (1 - \epsilon(1 + \lambda))\mu(I \setminus J)$. For $\epsilon$ sufficiently small, it follows that $(1 - \epsilon(1 + \lambda)) + (1 - \epsilon(1 + \lambda)) \ge 1$, so by the pigeonhole principle, $\mu(S^{h_k}A \cap A) > 0$. On $J$, $\frac{d\mu \circ S^{h_k}}{d\mu} = \lambda$, so it follows that the system is type $\text{III}_\lambda$.  The proof of ergodicity is as in \cite[4.2.1]{RuSi89}.
\end{proof}

We adapt the argument in \cite[Theorem 1.5]{AdFrSi97} to the nonsingular setting to  show that $(X, S, \mathcal{B}, \mu)$ is in fact not weakly doubly ergodic; this argument was used to prove that the infinite measure-preserving  HK(+1) transformation is not doubly ergodic but it also yields that is is not weakly doubly ergodic.
\begin{thm}
Let $0<\lambda <1$ and let $S$ be the  HK$(+1,\lambda)$ nonsingular transformation. 
There exists subsets $A$ and $B$ of $X$ of positive measure such that for all $n\in\mathbb Z$,
$$\mu \times \mu((S \times S)^n(A \times A) \cap (A \times B)) = 0.$$
\end{thm} 

\begin{proof}
Let $B$ be the top level of $C_1$ and $A = S^{-1}B$. Let $R_n$ be the transformation on $C_n$ that maps each level to the one above it and maps the top level to the bottom level. Note that $R_n$ agrees with $S$ on $C_n$ except on the top level. For $L = A$ or $L = B$, define
$$I_n(A, L) = \{i: 0 \le i < h_n, \mu(R_n^i A \cap L) > 0\}$$
$$I_n = I_n(A, A) \cap I_n(A, B).$$
We show by induction that $I_n = \emptyset$. Since $I_1(A, B) = \emptyset$, $I_1 = \emptyset$. Assume that $I_n = \emptyset$. Note that $A$ and $B$ are unions of levels in $C_n$. As $A$ is not contained in any of the additional spacers added to the system on the $n$th iteration of the cutting/stacking process (for $n > 1$), $R_{n + 1}^{2h_n}A$ and $R_{n + 1}^{2h_n + 1}A$ are contained in the spacers placed on the right subcolumn of $C_n$. We thus have the following:
$$I_{n + 1}(A, L) \subset I_n(A, L) \cup (I_n(A, L) + h_n) + (I_n(A, L) + 2h_n + 1) + (I_n(A, L) + 3h_n + 1).$$
Hence,
\begin{align*}
I_{n + 1} &= I_{n + 1}(A, A) \cap I_{n + 1}(A, B) \\
&\subset I_n(A, A) \cap I_n(A, B) \\
&\cup (I_n(A, A) + h_n) \cap (I_n(A, B) + h_n) \\
&\cup (I_n(A, A) + 2h_n + 1) \cap (I_n(A, B) + 2h_n + 1) \\
&\cup (I_n(A, A) + 3h_n + 1) \cap (I_n(A, B) + 3h_n + 1) \\
\end{align*}
By induction, all rows are empty. For all $i$, there exists $n > 0$ such that $S^iA = R_n^iA$, so $\mu \times \mu((S \times S)^i(A \times A) \cap (A \times B)) = 0$.
\end{proof}
Hence, there exists $A$ and $B$ of positive measure such that there does not exist $n$ such that $\mu(S^nA \cap A) > 0$ and $\mu(S^nA \cap B) > 0$. This shows that $S$ is not weakly doubly ergodic. Now we show using a similar argument to \thref{th1} that $S$ is EIC.
\begin{thm}
Let $(X, S, \mathcal{B}, \mu)$ denote the nonsingular type $\text{III}_\lambda$ HK$(+1,\lambda)$ transformation. Then there does not exist a nontrivial Borel equivariant map $\phi: (X, S, \mathcal{B}, \mu) \to (Y, T, d)$ where $T$ is an isometry on a separable metric space $Y$.
\end{thm}
\begin{proof}
By \thref{dense}, there exists $y$ in the image of $\phi$ that has a dense orbit under $T$ in the image of $\phi$ in $Y$. There exists a point $q \in Y$ such that the set
\[A = \{x \in X: d(\phi(x), q) < \frac{\epsilon}{2} \}\]
has positive measure. Let $L$ be a level in column $n$ for which $\mu(A \cap L) > (1 - \frac{1}{1000(\lambda + 1)})\mu(L)$. Since $\mu(S^{h_n}L \cap L) \ge \frac{1}{\lambda + 1}\mu(L)$, we must have $\mu(S^{h_n}A \cap A) > 0$. Thus, there exists some $z \in A$ such that $S^{h_n}z\in A$, so \[d(T^{h_n}(\phi(z)), \phi(z)) \le d(\phi(z), q) + d(T^{h_n}(\phi(z)), q) < \epsilon.\] Next, choose $\ell$ such that 
\[
d(T^\ell(y), \phi(z)) < \epsilon.\] We then have 
\begin{align*}
d(T^{h_n} y, y) = d(T^{h_n + \ell}y, T^\ell y) &\le d(T^{h_n + \ell}y, T^{h_n}\phi(z)) + d(T^{h_n}\phi(z), \phi(z)) + d(\phi(z), T^\ell y)\\
 &< 3\epsilon.\end{align*}  Now for each  $m \ge n$ there exists a level $L'$ in $C_m$  such that $\mu(A \cap L') > (1 - \frac{1}{1000(\lambda + 1)})\mu(L')$. As a result, we must have $d(T^{h_m}y, y) < 3\epsilon$ for all $m \ge n$. Then we have $d(T^{4h_n + 1}y, y) < 3\epsilon$, so by the triangle inequality, we have
\[d(T^{4h_n}y, y) \le d(T^{h_n}y, y) + d(T^{2h_n}y, T^{h_n}y) + d(T^{3h_n}y, T^{2h_n}y) + d(T^{4h_n}y, T^{3h_n}y) < 12\epsilon.\]
So
\[d(Ty, y) \le d(T^{4h_n}y, y) + d(T^{4h_n + 1}y, y) < 15\epsilon.\]
This proves that $y$ is a fixed point under $T$,  so the  map is trivial. Thus the type $\text{III}_\lambda$ $4h_k + 1$ transformation is EIC.
\end{proof}

One can also obtain type $\text{III}_1$ examples. In our construction, at even times of the inductive construction,   choose a fixed $\lambda_1$ and at odd times choose a  $\lambda_2$ such that $\log(\lambda_1)/\log(\lambda_2)$ is irrational; this results  in a 
type $\text{III}_1$ transformation by a standard argument (see \cite{DaSi09}), and our arguments can also be adapted to show these transformations are EIC and not WDE.

\section{WM and EIC}\label{rigidity}
As we have mentioned, in \cite{GlWe16}, Glasner and Weiss asked whether there exists a nonsingular (or infinite measure)  transformation that is WM but not EIC. In this section we consider results inspired by this question. \\\\
In this section, we first show that transformations admitting isometric factors to locally compact metric spaces are not weakly mixing. Next, we will show that topological rigidity does not imply measure theoretic rigidity, ruling out an approach to proving the equivalence of non-equivalence of EIC and WM by only using topological rigidity. We will then construct an isometric factor for a family of tower transformations. We will then prove using either \thref{locallycompact} or the equivalence of EUC and WM that a subset of those tower transformations are not weakly mixing. Finally, we'll use that result and \cite{AaNa87} and \cite[Chapter 15]{Na98} to deduce that there exists irrational rotations satisfying certain rigidity conditions. A similar analysis of the latter two processes described can be found in \cite{AARONSON_2013} and \cite{BERGELSON_2013}. 

Definition \ref{D:lceic}, as well as Proposition~\ref{locallycompact}  were  introduced in unpublished undergraduate thesis \cite{Bea} for the case of sigma-finite subinvariant transformations (that include infinite measure-preserving transformations).  We state the proposition for  the setting of nonsingular transformations  and include a proof   essentially following the ideas in \cite{Bea}.

\begin{defn}\label{D:lceic}
A nonsingular system $(X, \mu, T, \mathcal{B})$ is \textbf{Locally Compact Ergodic with Isometric Coefficients (LCEIC)} if 
for any locally compact  metric space $(Y,d)$ and $S : Y \to Y$ an invertible isometry,  for any equivariant Borel  map  $\phi : (X, \nu,T) \to Y$,  i.e., $S \circ \phi = \phi \circ T$ $\mu$-a.e., the image of almost all points in $X$ under $\phi$ is a single point., i.e., the map $\phi$ is constant a.e. 
\end{defn}

\begin{prop}\thlabel{locallycompact} 
For  nonsingular transformations, LCEIC is equivalent to WM.
\end{prop}

\begin{proof} We will first show that WM implies LCEIC. Let $(X, \nu,T, \mathcal{B})$ be  a nonsingular transformation that is WM. From the definition of WM it follows that every $T$-invariant  L$^\infty$ function is constant a.e., so $T$ is ergodic. Let  $(Y, d, S)$ denote an invertible isometry on a locally compact metric space. Let $f: X \to Y$ be a Borel equivariant map.  Assume that there does not exist one point $p \in Y$ such that $f^{-1}(p)$ is of  full measure. By \thref{dense}, we may assume that there exists a point $y \in Y$ with a dense orbit on an invariant set of full measure. 
For each $a \in Y$, there exists a sequence $n_{i, a}$ such that 
$$\lim_{i \to \infty} S^{n_{i, a}}(y) = a$$
We define a group operation 
$$a \cdot b = \lim_{i \to \infty} S^{n_{i, a} + n_{i, b}}(y)$$
It is clear that the limit exists. To check that the operation is well defined, if there were two sequences $m_{i, a}$ and $n_{i, a}$ such that $a = \lim_{i \to \infty} S^{n_{i, a}}y = \lim_{i \to \infty} S^{m_{i, a}}(y)$, then letting 
$$e = \lim_{i \to \infty} S^{n_{i, a} - m_{i, a}}y$$
we can check that    
$$a \cdot e = a.$$
This implies that $e$ is the identity. Thus, we have that $S$ is a rotation on a locally compact abelian group $Y$. Thus, by Pontryagin duality, there exists a character $\chi$ such that $\chi(Sy) \neq 1$. In addition, note that $\chi(Sa) = \chi(Sy \cdot a) = \chi(Sy) \cdot \chi(a)$ for each $a \in Y$. As $\chi$ is continuous, $\chi$ is measurable with respect to $\nu$. Hence $\chi(Sy)$ is a nontrivial eigenvalue of $S$ with eigenfunction $\chi$. Let $g = \chi \circ f$. Then letting $\lambda = \chi(Sy)$, we have
$$\lambda g = \lambda \chi \circ f = \chi \circ S \circ f = \chi \circ f \circ T = g \circ T$$ 
Hence, $\lambda$ is an eigenvalue for $T$ and $T$ is not WM, a contradiction. \\\\

Now we will show that LCEIC implies WM. Suppose $(X, \nu, T, \mathcal{B})$ is a nonsingular system that is not weakly mixing. Then there exists a Borel eigenfunction $f: X \to \mathbb S^1$ with $f \circ T = \lambda f$ pointwise almost everywhere for some $\lambda \in \mathbb S^1$. Letting $(\mathbb S^1, d, S)$ denote the rotation by $\lambda$, it follows that $f: X \to \mathbb S^1$ is a Borel equivariant map to a compact isometry. Hence, $(X, \nu, T, \mathcal{B})$ is not LCEIC.
\end{proof}
Let $(M, \rho)$ be a complete separable metric space and $T: M \to M$ a homeomorphism. If $(q_i)$ is a topological rigidity sequence for $T$, then it is well known that for any finite measure $\mu$ such that $\mu T = \mu$, $(q_i)$ is a measurable rigidity sequence for $T$. If $M$ is compact, then there exists a finite invariant measure, but in general, $T$ need not admit a finite invariant measure. Thus, if one finds a topological dynamical system $(M, T, \rho)$ that admits a topological rigidity sequence $(q_i)$ that is not a rigidity sequence for an irrational rotation, $T$ could still fail to be weakly mixing for any nonsingular measure $\nu$ on $M$. The below example illustrates that fact.
\begin{thm}
There exists a separable metric space $(Y, d)$ and an isometry $S$ on $Y$ such that $S$ is topologically uniformly rigid with respect to $(n_i)$ but for any finite Borel measure $\nu$ on $Y$ which is ergodic and nonsingular with respect to $S$, $(Y, S, d)$ is not measure theoretically rigid with respect to $(n_i)$.
\end{thm}
\begin{proof}
In \cite{FaKa}, it is shown that there exists a sequence of integers $(n_i)$ such that for any $\gamma \neq 1$, $\gamma^{n_i}$ does not tend to $1$ and that $n_i$ are rigidity times for an invertible measure preserving weakly mixing transformation $(X, T, \mu, \mathcal{B})$ (invertibility isn't shown explicitely in \cite{FaKa} but it follows from the construction being a Gaussian measure space construction). Let $A$ be a non-invariant positive measure subset of $X$. Then as $T$ is rigid with respect to $n_i$
$$\mu(T^{n_i}A \triangle A) \to 0.$$
Let $X_1 = \{T^nA: n \in \mathbb{Z}\}$. Define a metric $d$ on $X_1$ with $d(B, C) = \sqrt{\mu(B \triangle C)}$. Let $Y$ be completion of $X_1$ with respect to $d$. Let $S: X_1 \to X_1$ denote the transformation that takes $T^nA$ to $T^{n + 1}A$. Since 
$$d(S(T^{n}(A)), S(T^{m}(A))) = d(T^{n + 1}(A), T^{m + 1}(A)) = \sqrt{\mu(T^{n + 1}(A) \triangle T^{m + 1}(A))}$$
$$= \sqrt{\mu(T^nA \triangle T^mA)} = d(T^nA, T^mA),$$
it follows that $S$ is an isometry, so $S$ extends to a map on $Y$. Since $T$ is weakly mixing, it is in particular totally ergodic. The metric space $(X_1,d)$ has no isolated points so $Y$ is a complete metric space with no isolated points and by the Baire category theorem it is uncountable. We have that  $Y$ is uncountable with $A$ having dense orbit under $S$ and $(Y, S, d)$ is a topological dynamical system with $S$ an isometry. \\\\
We now show that $(Y, S, d)$ is uniformly rigid with respect to $(n_i)$. To do this, let $B$ be an element of $Y$. As $\{T^nA\}$ is dense in $Y$, there exists $\ell$ such that $d(T^\ell A, B) < \frac{\epsilon}{3}$. Next, pick $i$ large enough so that $d(T^{n_i}A, A) < \frac{\epsilon}{3}$. Then,
$$d(T^{n_i}B, B) \le d(T^{n_i + \ell}A, T^{n_i}B) + d(T^{n_i + \ell}A, T^{\ell}A) + d(T^\ell A, B) < \epsilon$$
so $T^{n_i}B \to B$ with respect to $d$. \\\\
There is another interpretation for the set $Y$. Observe that $X_1$ equipped with $d$ can be viewed as a subspace of $L^2(X, \mu)$ via the map $T^nA \mapsto 1_{T^nA}$. As $L^2(X, \mu)$ is complete and $X_1$ shares the same metric as the $L^2$ metric, $Y$ can be viewed as a closed subspace of $L^2(X, \mu)$. \\\\
Suppose there exists a finite nonsingular ergodic measure $\nu$ on $Y$ that is measure theoretically rigid with respect to $(n_i)$. Let $\mathcal{B}(Y)$ denote the Borel subsets of $Y$. As there exists a continuous isometry between $Y$ and a Hilbert space, namely $L^2(X, \mu)$, $(Y, \nu, S, \mathcal{B})$ is not ergodic with unitary coefficients. Hence, $S$ is not weakly mixing. Hence, there exists an eigenfunction $f \in L^\infty(Y, \nu)$ with eigenvalue $1 \neq \lambda \in \mathbb{C}^* \cong S^1$ with $f \circ S = \lambda f$. On the other hand, if $(n_i)$ is a measure theoretic rigidity sequence for $(S, Y, d)$, $\|f \circ S^{n_i} - f\|_{L^2(Y, \nu)} \to 0$. But
$$\|f \circ S^{n_i} - f \|_{L^2(Y, \nu)} = \|f\|_{L^2(Y, \nu)}|\lambda^{n_i} - 1| \to 0.$$
But as $n_i$ is not a rigidity sequence for any rotation, $|\lambda^{n_i} - 1|$ cannot possibly go to $0$. Hence, there does not exist a finite nonsingular ergodic measure on $Y$ that is rigid with respect to $S$ along the sequence $(n_i)$.
\end{proof}
\begin{rem}
The system $(Y, S)$ does not admit any invariant Borel probability measure $\eta$ as on any invariant probability measure $\eta$, uniform rigidity implies measure theoretic rigidity, and because ergodic measures are extreme points in the space of probability invariant measures. As any homeomorphism on a compact space admits an ergodic invariant measure, $Y$ is not a compact metric space.
\end{rem}
Another possible candidate of WM systems that are not EIC are rank-one transformations since the spectrum of rank-one transformations are somewhat understood by the results of Aaronson and Nadkarni in \cite{AaNa87} and \cite[Chapter 15]{Na98}:
\begin{thm}\thlabel{AaNa}\cite[Corollary 15.57]{Na98}
Let $T$ be a rank-one $\sigma$-finite measure preserving tower transformation with stack heights $h_i$. Then the spectrum of $T$ contains all complex numbers $\lambda$ satisfying
$$\sum_{i = 1}^\infty |\lambda^{h_i} - 1|^2 < \infty.$$
\end{thm}
\begin{proof}
We use the notation of \cite[Corollary 15.57]{Na98}. First, we pass from $T$ to an isomorphic rank-one transformation without any spacers in the last column as done in \cite[15.42]{Na98}. This is done by adding the spacers in the tower transformation to the first column instead of the last column. Then $m_k$ is the number of cuts of the stack and $\gamma_{k, i}$ is the height of each stack, including spacers in the $i$th column of the $k$th iteration of the cutting and stacking transformation. Hence, since $T$ is a tower transformation, $m_k = 2$ for all $k$ and $\gamma_{k, 1} = h_{k} - h_{k - 1}$ and $\gamma_{k, 2} = h_k$ as $\gamma_{k, i}$ are the heights of each stack in the cutting and stacking transformation . It then follows from \cite[Corollary 15.57]{Na98} that
$$\sum_{i = 1}^\infty |\lambda^{h_i} - 1|^2 < \infty.$$
\end{proof}
This suggests that if there exists a metric space $(M, d)$, an isometry $T: M \to M$ such that there exists a rank-one tower transformation whose heights $(h_i)$ satisfy the property that no complex numbers satisfy \thref{AaNa}, then the rank-one transformation is an EIC transformation that isn't weakly mixing. 
The remainder of this section will be using \thref{AaNa} to either rule out more examples or deduce rigidity results of rotations on the circle.
\begin{prop}\thlabel{Tower}
Let $(q_i)$ be a sequence of natural numbers with $\frac{q_{i + 1}}{q_i} \ge 2$. Let $(S, Y, d)$ be a topological dynamical system with $S$ an isometry. Suppose the metric space $Y$ satisfies $\sum d(S^{q_i}y, y) < \infty$ for some $y \in Y$. Then there exists a nontrivial equivariant map $\phi$ from a rank-one  tower transformation with heights $(q_i)$ to $Y$. In particular, the image of the map is in $\overline{\{S^n y\}}_{n \in \mathbb{N}}$.
\end{prop}
\begin{proof}
Define $(X, S, \mu, \mathcal{B})$ to be the rank-one tower transformation with heights $q_i$. We may pass to an equivalent finite measure $\nu$ such that $\nu(C_{n + 1} \setminus C_n) = r_n\mu(C_{n + 1} \setminus C_n)$ where $C_k$ denotes the $k$th column of the cutting and stacking construction of $X$. Define a function $g_k: C_k \to Y$ by $g_k(C_k^j) = S^j(y)$ where $C_k^j$ is the $j$th interval in column $C_k$ and $g_k$ is zero on $X \setminus C_k$. Define $f_k: X \to Y$ by $f_k = g_k$ on $C_k$ and $f_k = g_j$ on $C_j \setminus C_{j - 1}$ for all $j \ge k + 1$. We claim that $f_k$ converges in $L^\infty$ to some $\phi$ which has the desired properties. For this, note that on $X\setminus C_k$, $f_k= f_{k+1}$, and so we need to worry only about the difference on $C_k$. $C_{k+1}$ is built by cutting the levels of $C_k$ into $2$ columns of equal size, and so we will let $C_k^j(n)$ denote the $n$th level in the $j$th new column created by cutting $C_k$ in this fashion, where $0\leq j \leq 1$. Observe that on $C_k^j(n)$, $f_{k+1} = S^{jq_k+n}y$, and $f_k = S^ny$, giving that for each $0\leq j\leq 1$,
\begin{align*}d(f_{k+1}, f_k) &= d(S^n(S^{jq_k}y), S^ny)\\
&\le d(S^{q_k}y, y)
\end{align*}
Since
$$\sum d(S^{q_k}y, y) < \infty$$
the $(f_k)$'s form a Cauchy sequence and converge pointwise to some $\phi$. Since $\phi(Tx) = \lim_{k \to \infty} f_k(Tx) = \lim_{k \to \infty} Sf_k(x) = S\phi(x)$ we have that $\phi$ is a desired factor.
\end{proof}
Thus, combining the \cite{GlWe16} result that proves the equivalence between EUC and WM and \cite{AaNa87}, we obtain the following:
\begin{prop}\thlabel{Unitary}
Let $(T, X, \mu)$ be a $\sigma$-finite measure preserving transformation, $(q_i)$ a sequence of integers with $\frac{q_{i + 1}}{q_{i}} \ge 2$ that has a set $B$ such that
$$\sum_{i = 1}^\infty \sqrt{\mu(B \triangle T^{q_i}B)} < \infty.$$
Then there exists a complex number $\lambda$ such that
$$\sum_{i} |1 - \lambda^{q_i}|^2 < \infty.$$
\end{prop}
\begin{proof}
By \thref{Tower}, there exists a nontrivial equivariant  map from a rank-one tower transformation $S$ with heights $(q_i)$ to the metric space $\overline{\{\chi_{T^nB}\}}_{L^2(X, \mu)}$. This implies that $S$ is not ergodic with unitary coefficients and hence not weakly mixing. By \thref{AaNa}, it follows that
there exist a complex number $\lambda$ such that
$$\sum_{i} |1 - \lambda^{q_i}|^2 < \infty.$$

\end{proof}

Finally, we rule out a family of tower transformations with the quotient of two consecutive heights are bounded. To do this, we need the following lemma:
\begin{lemma}\thlabel{Compact}
Let $(Y, S, d)$ be a topological dynamical system with $Y$ an isometry. Let $(q_i)$ be a sequence with
$$2 \le \frac{q_{i + 1}}{q_i} \le K$$
where $K$ is some positive constant. Suppose there exists $y \in Y$ with $\sum_{i = 1}^\infty d(S^{q_i}y, y) < \infty$. Then $T := \{S^n y\}_{n \in \mathbb{N}}$ is totally bounded.
\end{lemma}
\begin{proof}
Let $\frac{q_{i + 1}}{q_i}$ be bounded by $K$ and choose $L$ large enough so that
$$\sum_{i = L}^\infty d(S^{q_i}y, y) < \frac{\epsilon}{K}.$$
We assert that $T$ is covered by $q_L - 1$ balls of radius $2\epsilon$ centered on $y, Sy, S^2y, \dots, S^{q_L - 1}y$. Let $A = \{n: n \in \mathbb{N} S^n y \in B_\epsilon(y)\}$. Let
$$H = \{h \in \mathbb{N}: h = h_1 + h_2 + \cdots + h_k, k \in \mathbb{N}, h_i = q_j, j \ge L, h_i \text{ are distinct}\}.$$
Notice that if $h \in H$, then
$$d(S^hy, y) \le \sum_{j = 1}^k d(S^{h_j}y, y) \le \sum_{i = L}^\infty d(S^{q_i}y, y) < \frac{\epsilon}{K}.$$
Let
$$P = \{p \in \mathbb{N}: p = p_1 + p_2 + p_3 + \cdots + p_K, p_i \in H \forall i\}.$$
The set $P$ is contained in $A$ because for $p \in P$,
$$d(S^py, y) \le \sum_{j = 1}^K d(S^{p_j}y, y) \le \epsilon.$$
Then the subset
$$B = \{b \in \mathbb{N}: b = a\sum_{j = L}^\infty a_j q_{j}, 0 \le a_i \le \lfloor q_{i + 1}/q_{i} \rfloor, \text{ all but finitely many } a_j \text{ are } 0 \} \subset P.$$
is contained in $A$. Let $r$ be the least integer in $B$ greater than a given $b \in B$. We claim that $b - r \le q_L$. Write
$$b = \sum_{j = L}^\infty a_j q_{j}$$
where we use similar notation as before. Pick the least $\ell$ with $a_\ell \neq \lfloor q_{\ell + 1}/q_\ell \rfloor$. This exists since all but finitely many $a_j$ are zero. Then
$$r \le s := (a_{\ell} + 1) q_\ell + \sum_{j = \ell + 1}^\infty a_j q_j \le b + q_L$$
since
$$b + q_L =  (a_L + 1)q_L + \cdots + a_\ell q_\ell + \cdots \ge (a_{L + 1} + 1)q_{L + 1} \cdots + a_\ell q_\ell + \cdots \ge \cdots \ge (a_\ell + 1)q_\ell + \cdots.$$
Hence $r - b \le q_L$ as claimed. The difference between two consecutive elements in $A$ is less than or equal to $q_L$. Since $B_\epsilon(S^jy) \supseteq j + A$ (since $S$ is an isometry), it follows that the balls of radius $\epsilon$ around $y, Sy, \dots, S^{q_L - 1}y$ covers $T$.
\end{proof}
Combining \thref{Tower} with \thref{Compact}, we have the following:

\begin{thm}\thlabel{Eigencoro}
Let $(Y, d, S)$ be a topological dynamical system with $S$ an isometry. Suppose $(q_i)$ is a sequence of integers such that $\frac{q_{i + 1}}{q_i} < K$ for some $K$ and there exists some $y \in Y$ such that $\sum_{i = 1}^\infty d(S^{q_i}y, y) < \infty$. Then there exists $\lambda \in \mathbb{T}$ such that $\sum_{i = 1}^\infty |1 - \lambda^{q_i}|^2 < \infty$. In particular, the tower transformation with heights $(q_i)$ is not weakly mixing.
\end{thm}

\begin{proof}
By \thref{Tower}, there exists a nontrivial equivariant Borel map from a rank-one tower transformation having heights $(q_i)$ to $(Y, S, d)$. The image of the  map is in particular contained in $C := \overline{\{S^ny\}}_{n \in \mathbb{N}}$. By, \thref{Compact} $C$ is compact, so the tower transformation  is not weakly mixing by \thref{locallycompact}. Thus, by \thref{AaNa}, $\sum_{i = 1}^\infty |1 - \lambda^{q_i}|^2 < \infty$.
\end{proof}

\section{Open Questions}
\begin{qn}
Does there exist $(n_i)$ be a sequence such that $\frac{n_{i + 1}}{n_i}$ is unbounded and $(n_i)$ is not a rigidity sequence for any irrational number $\alpha$ but there exists a topological dynamical system $(Y, S, d)$ with $S$ an isometry and an element $x \in Y$ such that $\sum_{i = 1}^\infty d(S^{n_i}x, x) < \infty$?
\end{qn}
\begin{rem}
$\frac{n_{i + 1}}{n_i}$ cannot go to infinity because otherwise, as observed by \cite{BERGELSON_2013} and \cite{ET}, $(n_i)$ would be a rigidity sequence for uncountably many irrational numbers.
\end{rem}
\begin{qn}
Can one improve \thref{Eigencoro} so that $\sum_{i = 1}^\infty |1 - \lambda^{q_i}| < \infty?$
\end{qn}

\bibliographystyle{plain}
\bibliography{NonsingularBib}

\begin{thebibliography}{10}

\bibitem{AARONSON_2013}
Jon Aaronson, Maryam Hosseini, and Mariusz Lemanczyk.
\newblock Ip-rigidity and eigenvalue groups.
\newblock {\em Ergodic Theory and Dynamical Systems}, 34(4):1057--1076, Mar
  2013.

\bibitem{AaNa87}
Jon Aaronson and Mahendra Nadkarni.
\newblock {$L_\infty$} eigenvalues and {$L_2$} spectra of nonsingular
  transformations.
\newblock {\em Proc. London Math. Soc. (3)}, 55(3):538--570, 1987.

\bibitem{AaLiWe79}
Jonathan Aaronson, Michael Lin, and Benjamin Weiss.
\newblock Mixing properties of {M}arkov operators and ergodic transformations,
  and ergodicity of {C}artesian products.
\newblock {\em Israel J. Math.}, 33(3-4):198--224 (1980), 1979.
\newblock A collection of invited papers on ergodic theory.

\bibitem{AdFrSi97}
Terrence Adams, Nathaniel Friedman, and Cesar~E. Silva.
\newblock Rank-one weak mixing for nonsingular transformations.
\newblock {\em Israel J. Math.}, 102:269--281, 1997.

\bibitem{AdSi18}
Terrence Adams and Cesar~E. Silva.
\newblock Weak mixing for infinite measure invertible transformations.
\newblock In {\em Ergodic theory and dynamical systems in their interactions
  with arithmetics and combinatorics}, volume 2213 of {\em Lecture Notes in
  Math.}, pages 327--349. Springer, Cham, 2018.

\bibitem{Eigenoperators}
Anatole Beck.
\newblock Eigen operators of ergodic transformations.
\newblock {\em Transactions of the American Mathematical Society},
  94(1):118--129, 1960.

\bibitem{BERGELSON_2013}
Vitaly Bergelson, Andres Del~Junco, Mariusz Lemanczyk, and Joseph Rosenblatt.
\newblock Rigidity and non-recurrence along sequences.
\newblock {\em Ergodic Theory and Dynamical Systems}, 34(5):1464--1502, Apr
  2013.

\bibitem{BoFiSi01}
Amie Bowles, Lukasz Fidkowski, Amy~E. Marinello, and Cesar~E. Silva.
\newblock Double ergodicity of nonsingular transformations and infinite
  measure-preserving staircase transformations.
\newblock {\em Illinois J. Math.}, 45(3):999--1019, 2001.

\bibitem{DaSi09}
Alexandre~I. Danilenko and Cesar~E. Silva.
\newblock Ergodic theory: non-singular transformations.
\newblock In {\em Mathematics of complexity and dynamical systems. {V}ols.
  1--3}, pages 329--356. Springer, New York, 2012.

\bibitem{ET}
P.~Erdos and S.J. Taylor.
\newblock On the set of points of convergence of a lacunary trigonometric
  series and the equidistribution properties of related sequences.
\newblock {\em Proceedings of the London Mathematical Society}, s3-7:598--615,
  1957.

\bibitem{FaKa}
Bassam Fayad and Adam Kanigowski.
\newblock Rigidity times for a weakly mixing dynamical system which are not
  rigidity times for any irrational rotation.
\newblock {\em Ergodic Theory Dynam. Systems}, 35(8):2529--2534, 2015.

\bibitem{GlWe93}
Eli Glasner and Benjamin Weiss.
\newblock Sensitive dependence on initial conditions.
\newblock {\em Nonlinearity}, 6(6):1067--1075, 1993.

\bibitem{GlWe16}
Eli Glasner and Benjamin Weiss.
\newblock Weak mixing properties for non-singular actions.
\newblock {\em Ergodic Theory Dynam. Systems}, 36(7):2203--2217, 2016.

\bibitem{Bea}
Beatrix Haddock.
\newblock Multiplier properties in infinite measure.
\newblock {\em Undergraduate Thesis, Williams College}, 2018.

\bibitem{KaPa63}
S.~Kakutani and W.~Parry.
\newblock Infinite measure preserving transformations with ``mixing''.
\newblock {\em Bull. Amer. Math. Soc.}, 69:752--756, 1963.

\bibitem{MRS99}
E.~J. Muehlegger, A.~S. Raich, C.~E. Silva, M.~P. Touloumtzis, B.~Narasimhan,
  and W.~Zhao.
\newblock Infinite ergodic index {${\bf Z}^d$}-actions in infinite measure.
\newblock {\em Colloq. Math.}, 82(2):167--190, 1999.

\bibitem{Na98}
M.~G. Nadkarni.
\newblock {\em Spectral theory of dynamical systems}.
\newblock Birkh\"auser Advanced Texts: Basler Lehrb\"ucher. [Birkh\"auser
  Advanced Texts: Basel Textbooks]. Birkh\"auser Verlag, Basel, 1998.

\bibitem{RuSi89}
Daniel~J. Rudolph and Cesar~E. Silva.
\newblock Minimal self-joinings for nonsingular transformations.
\newblock {\em Ergodic Theory Dynam. Systems}, 9(4):759--800, 1989.

\bibitem{Si08}
C.~E. Silva.
\newblock {\em Invitation to ergodic theory}, volume~42 of {\em Student
  Mathematical Library}.
\newblock American Mathematical Society, Providence, RI, 2008.

\end{thebibliography}

\end{document}